\documentclass[12pt]{amsart}

\usepackage[margin = 1.3in]{geometry}

\usepackage{tikz}
\usepackage{tikz-cd}
\usepackage{url, hyperref}
\usepackage{float}

\tikzset{->-/.style={decoration={
  markings,
  mark=at position .45 with {\arrow{>}}},postaction={decorate}}}
  
\usetikzlibrary{shapes.arrows}
\usetikzlibrary{decorations.pathreplacing}
\usepackage[all]{xy}

\tikzset{->-/.style={decoration={
  markings,
  mark=at position .45 with {\arrow{>}}},postaction={decorate}}}

\usepackage{amsmath}
\usepackage{amssymb}
\usepackage{enumitem}
\usepackage{graphicx}
\usepackage{mathdots}
\usepackage{color}
\usepackage{diagbox}
\usepackage{array, makecell}
\usepackage{rotating}
\usepackage{amsmath}
\usepackage{tikz-cd}

\def\cO{\mathcal{O}}

\def\cM{{\mathcal{M}}}

\def\C{\mathbb{C}}

\def\b1{{\bf 1}}

\def\P{\mathbb{P}}
\def\bP{\mathbb{P}}

\def\HH{{\mathcal{H}}}

\def\cE{\mathcal{E}}
\def\cF{\mathcal{F}}
\def\cL{\mathcal{L}}

\newcount\colveccount
\newcommand*\colvec[1]{
        \global\colveccount#1
        \begin{pmatrix}
        \colvecnext
}
\def\colvecnext#1{
        #1
        \global\advance\colveccount-1
        \ifnum\colveccount>0
                \\
                \expandafter\colvecnext
        \else
                \end{pmatrix}
        \fi
}

\usepackage{amsthm}
\newtheorem{definition}{Definition}[section]
\newtheorem{theorem}[definition]{Theorem}

\newtheorem{lemma}[definition]{Lemma}

\newcommand{\Pic}{\mathsf{Pic}}

\newcommand{\Bl}{{\mathsf{Bl}}}

\newcommand{\Tev}{{\mathsf{Tev}}}

\newcommand{\rk}{\text{rk}}

\newcommand{\ev}{{\text{ev}}}

\usepackage{amsmath,calligra,mathrsfs}
\DeclareMathOperator{\cHom}{\mathscr{H}\text{\kern -3pt {\calligra\large om}}\,}

\def\cE{\mathcal{E}}
\def\cF{\mathcal{F}}
\def\cL{\mathcal{L}}

\subjclass[2020]{14H10, 14H60, 14J26, 14N10, 14N35}

\title{Curves on Hirzebruch surfaces and semistability}

\author{Alessio Cela}
\address{ University of Cambridge, Department of pure mathematics and mathematical statistics
\hfill \newline\texttt{}
 \indent Centre for Mathematical Sciences, Wilberforce Road Cambridge, UK} \email{{\tt ac2758@cam.ac.uk}}

\author{Carl Lian}

\address{Tufts University, Department of Mathematics, 177 College Ave
\hfill \newline\texttt{}
 \indent Medford, MA 02155} \email{{\tt Carl.Lian@tufts.edu}}

\date{\today}

\usepackage{graphicx}

\begin{document}

\maketitle

\begin{abstract}
When $a\ge2$, we show that a general pointed curve never interpolates through the expected number of points in the Hirzebruch surface $\HH_a$, with one exception. In the exceptional case, the number of such interpolating maps is determined by the geometric Tevelev degrees of $\P^1$, which have been previously computed.
\end{abstract}

\section{Introduction}

Let $\cM_{g,n}$ be the Deligne-Mumford moduli space of genus $g$, $n$-pointed curves. Let $(C,p_1\ldots,p_n)$ be a general, pointed curve of genus $g$, by which we mean a general point of $\cM_{g,n}$. Let $X$ be a smooth, projective variety. Can $C$ be interpolated through $n$ general points of $X$? If so, then in how many ways?

More precisely, let $\beta\in H_2(X,\mathbb{Z})$ be an effective curve class, and let $\cM_{g,n}(X,\beta)$ be the moduli space of pointed maps to $X$ in class $\beta$. Let
\begin{equation*}
    \tau:\cM_{g,n}(X,\beta)\to \cM_{g,n} \times X^n
\end{equation*}
be the forgetful morphism. We assume that $2g-2+n>0$. One expects that $\tau$ is generically finite whenever
\begin{equation}\label{dim_constraint}
    \beta\cdot K_X^\vee=\dim(X)\cdot (n+g-1).
\end{equation}
We assume \eqref{dim_constraint} throughout.


\begin{definition}\label{def:tev}
    Suppose that every irreducible component $Z\subset \cM_{g,n}(X,\beta)$ dominating $\cM_{g,n} \times X^n$ is generically smooth of the expected dimension. Then, the \emph{geometric Tevelev degree} $\Tev^X_{g,n,\beta}$ is by definition the degree of $\tau$. 
\end{definition}

If no such components $Z$ exist, then $\Tev^X_{g,n,\beta}=0$ vacuously. The name ``geometric Tevelev degree'' was introduced in \cite{cps} after a result of Tevelev \cite{tev}, but in fact, the corresponding \emph{virtual} invariants go back to work on Vafa-Intriligator formulas, e.g. \cite{bdw,st,mo}. See \cite{bp,Cela2023} for more recent work on the virtual invariants. The geometric invariants, which are by definition enumerative counts as opposed to virtual intersection numbers, tend to be more subtle to compute, see \cite{cps,fl,cl1,lian_hyp,cl2,lian_pr} for recent computations and \cite{lian_pand,bllrst} for comparisons of virtual and geometric invariants.

We also say that $\Tev^X_{g,n,\beta}$ is \emph{well-defined} if the hypothesis of Definition \ref{def:tev} holds. Given a dominating component $Z\subset \cM_{g,n}(X,\beta)$ that is generically smooth of the expected dimension, a general curve interpolates through $n$ general points in $X$ via a map corresponding to a point of $Z$. Furthermore, this map admits no first-order deformations fixing $(C,p_1,\ldots,p_n)$ and the given points on $X$. If all dominating components $Z$ are generically smooth of the expected dimension, then the total number of interpolating maps is given by $\Tev^X_{g,n,\beta}$.

A basic obstruction to the existence of dominating components $Z\subset \cM_{g,n}(X,\beta)$ is that, if $f:C\to X$ is a general point of $Z$, then $f^{*}T_X$ must be semistable, see Lemma \ref{lemma: stability}. Heuristically speaking, if a curve on $X$ is more prone to deform in some tangent directions on $X$ than others, then it has insufficient room to interpolate through $n$ general points when maximally constrained. In particular, if $T_X$ is not $\beta$-semistable (see Definition \ref{def: stability}), then $\Tev^X_{g,n,\beta}$, if well-defined, must equal zero. There is a large body of work on (semi)stability of vector bundles and their restrictions to curves, see for instance \cite{gm,sacchiero,ramella,ran,larson,aly,clv,lv,lrt1,lrt2}.

In this note, we determine the geometric Tevelev degrees of Hirzebruch surfaces $\HH_a=\P_{\P^1}(\cO(-a)\oplus\cO)$ with $a\ge 2$. We first observe that, with one exception, the semistability obstruction is already enough to rule out the existence of any dominating components $Z\subset \cM_{g,n}(\HH_a,\beta)$.

\begin{theorem}\label{tev_hirzebruch_zero}
    Let $a\ge 2$ be an integer, and let $\HH_a=\P_{\P^1}(\cO(-a)\oplus\cO)$ be a Hirzebruch surface. Let $C_0\subset\HH_a$ denote the zero-section of $\HH_a$, corresponding to the map $\cO(-a) \to \cO(-a)\oplus \cO$. Suppose that $n\ge2$ (or that $g=0$ and $n\ge 3$), and that \eqref{dim_constraint} holds.
    
    Then, no irreducible component $Z\subset \cM_{g,n}(\HH_a,\beta)$ dominates $\cM_{g,n} \times X^n$ under the map
\begin{equation*}
    \tau:\cM_{g,n}(\HH_a,\beta)\to \cM_{g,n} \times \HH_a^n,
\end{equation*}
unless $a=2$ and $\beta= m [C_0]$ is a multiple of the zero-section $C_0\subset \HH_a$. In particular, outside of this case, we have $\Tev^{\HH_a}_{g,n,\beta}=0$.
\end{theorem}

That is, when $a\ge 2$, unless $a=2$ and $\beta=m[C_0]$, a general pointed curve admits no interpolations in class $\beta$ through the maximal expected number of points. This leaves only one interesting case.

\begin{theorem}\label{tev_h2}
    We have
    \begin{equation*}
        \Tev^{\HH_2}_{g,n,m[C_0]}=\Tev^{\bP^1}_{g,n,m}
    \end{equation*}
    for any $g,n,m$ for which $n=2m-g+1\ge3$.
\end{theorem}

The condition $n=2m-g+1$ is \eqref{dim_constraint} on both sides. If $2m-g+1<3$ and $m>0$, then it follows from the Brill-Noether theorem that there are no maps from a general curve to $\P^1$ or $\HH_2$ in the required homology classes. 

The geometric Tevelev degrees of $\HH_2$ are therefore determined by those of $\P^1$, which have been computed in \cite{cps,fl,cl1}. When $a=0$, the geometric Tevelev degrees of $\P^1\times\P^1$ are trivially determined by those of $\P^1$. The most difficult case is $a=1$, in which $\HH_1\cong\Bl_q(\P^2)$; some partial results are given in \cite{cl2}. Genus 0 invariants of $\HH_a$ in the log setting (with imposed tangencies) are computed in \cite{cil}.

\subsection{Conventions}

\begin{itemize}
\item We work over $\C$ for convenience, though Lemma \ref{lemma: stability} holds in arbitrary characteristic.
\item The projective bundle $\P(\cE)$ is the moduli space of \emph{lines} in the fibers of $\cE$.
\item Angle brackets $\langle - \rangle$ denote linear span.
\end{itemize}

\subsection{Acknowledgments}

We thank Roya Beheshti, Gavril Farkas, Eric Larson, Brian Lehmann, Aitor Iribar L\'{o}pez, Rahul Pandharipande, Eric Riedl, Jason Starr, Sho Tanimoto, Richard Thomas, and Isabel Vogt for conversations related to this note, and the anonymous referee for their helpful comments. A.C. is supported by SNF grant P500PT-222363. C.L. has been supported by NSF Postdoctoral Fellowship DMS-2001976 and an AMS-Simons travel grant.

\section{Semistability and interpolation}

\begin{definition}\label{def: stability}
Let $X$ be a smooth projective variety. Let $\mathcal{E}$ be a vector bundle on $X$ and $\beta \in H_2(X)$ be an effective curve class. We say that a non-zero sub-bundle $\mathcal{F} \subsetneq \mathcal{E}$ is \emph{$\beta$-destabilizing} if
\begin{equation*}
        \frac{c_1(\mathcal{F}) \cdot \beta }{\mathrm{rank}(\mathcal{F})} > \frac{c_1(\mathcal{E}) \cdot \beta }{\mathrm{rank}(\mathcal{E})}.
\end{equation*}
We say that $\cE$ is \emph{$\beta$-semistable} if there does not exist a non-zero $\beta$-destabilizing sub-bundle.

If $X$ is a curve, then we take implicitly $\beta=[X]$, in which case one recovers the standard notion of semistability. 
\end{definition}

The following lemma is well-known.

\begin{lemma}\label{lemma: stability}
Let $Z\subset \cM_{g,n}(X,\beta)$ be a component dominating $\cM_{g.n}\times X^n$. Assume that $Z$ is generically smooth of the expected dimension, and that \eqref{dim_constraint} holds. Then, for a general point $[f:C\to X]$ of $Z$, the restricted tangent bundle $f^{*}T_X$ is semistable.
\end{lemma}

\begin{proof}
Let $(C,p_1,\ldots,p_n)$ be a general point of $\cM_{g,n}$, and let $\tau_C:\cM_\beta(C,X)\to X^n$ be the pullback of $\tau$ over $[(C,p_1,\ldots,p_n)]\in\cM_{g,n}$. Let $x_i=f(p_i)$ for each $i=1, \ldots, n$. Then, the differential of $\tau_C$ at $[f]$
     \begin{equation*}
        (d\tau_C)_f:H^0(C,f^{*}T_X)\to (T_{x_i}X)^{n}    
    \end{equation*}
    is an isomorphism.
    
     Suppose that there exists a destabilizing subbundle $\cF\subset f^*T_X$. The image subspace $H^0(C,\cF)\subset H^0(C,f^{*}T_X)$ has image of dimension at most $n\cdot\rk(\cF)$, corresponding to the tangent vectors given by the subspaces $\cF_{p_i}\subset T_{x_i}X$. On the other hand, by Riemann-Roch,
     \begin{align*}
        h^0(C,\cF) &\ge \int_C c_1(\cF)+\rk(\cF)(1-g)\\
 &>\frac{\rk(\cF)}{\dim(X)} \int_C c_1(f^*T_X) +\rk(\cF)(1-g)\\
         &=\frac{\rk(\cF)}{\dim(X)} \int_Xc_1(T_X)\cdot\beta +\rk(\cF)(1-g)\\
         &=n\cdot\rk(\cF),
     \end{align*}
     which is a contradiction.
\end{proof}

In particular, if $\mathsf{Tev}^X_{g,n,\beta}$ is well-defined and non-zero, then $f^*T_X$ must be semistable for all $[f]$ enumerated by $\mathsf{Tev}^X_{g,n,\beta}$. The tangent bundle $T_X$ must also be $\beta$-semistable on $X$, or else a destabilizing sub-bundle would also be destabilizing upon pullback by a map $f:C\to X$ enumerated by $\mathsf{Tev}^X_{g,n,\beta}$. For example:
\begin{itemize}
    \item When $X=\bP^r$, the geometric Tevelev degrees $\Tev^{\P^r}_{g,n,d}$ are fully determined in \cite{lian_pr}, and can be checked to be non-zero. Therefore, the corresponding restricted tangent bundles $f^{*}T_{\P^r}$ are semistable. This had previously been obtained by E. Larson: \cite[Corollary 1.3]{larson} implies that, for a general point $f\in \cM_{g}(\P^r,d)$ of the unique (Brill-Noether) component of $\cM_{g}(\P^r,d)$ dominating $\cM_g$, the restricted tangent bundle $f^{*}T_{\P^r}$ is semistable whenever $\rho(g,r,d)=g-(r+1)(g-d+r)\ge0$ and $d$ is divisible by $r$. The condition $\rho(g,r,d)\ge0$ is needed for such a dominating component to exist, and the condition that $d$ is divisible by $r$ is necessary for $\Tev^{\bP^r}_{g,n,d}$ to be well-defined, in order to apply Lemma \ref{lemma: stability}. Note also that, when $g=0$, the restricted tangent bundle $f^{*}T_{\P^r}$ can only be semistable when $d$ is divisible by $r$, for arithmetic reasons.

    \item When $X$ is a hypersurface of degree at most $(r+4)/3$ in $\bP^{r+1}$ and $\beta$ is sufficiently large, then virtual Tevelev degrees of $X$ are enumerative (and thus equal to the geometric degrees) \cite[Theorem 1.12(2)]{bllrst}, and non-zero \cite[Theorem 1.4, Theorem 1.5]{bp}. Therefore, for any general point of a component of $\cM_{g,n}(X,d)$ dominating $\cM_{g,n}\times X^n$, the restricted tangent bundle $f^{*}T_{X}$ is semistable for $f:C\to X$. While tangent bundles of smooth hypersurfaces are known to be semistable \cite[Theorem 2]{pw}, the stronger statement for \emph{restricted} tangent bundles seems not previously to have been known.

    \item Calculations of \cite{cl2} imply similar semistability statements for toric blow-ups of $\bP^r$. For example, the genus 0 counts $\Tev^X_{0,n,\beta}$ for $X$ equal to the blow-up of $\bP^r$ at $r$ points \cite[Theorem 12]{cl2} are manifestly non-zero for $\beta$ in an appropriate range.
\end{itemize}

Lemma \ref{lemma: stability} may also be used in the opposite direction. It was once conjectured (e.g. \cite[Conjecture 0.1]{Kanemitsu}) that if $X$ is a Fano manifold with Picard number $1$, then its tangent bundle is (semi)stable. This has been shown to be false in \cite[Theorem 0.4]{Kanemitsu}. In particular, geometric Tevelev degrees of such Fano manifolds must be zero (when well-defined).

\section{Hirzebruch surfaces}

\begin{lemma}\label{lem: tev well-defined}
    Let $\HH_a=\P_{\P^1}(\cO(-a)\oplus\cO)$ be a Hirzebruch surface, with $a\ge 2$. Suppose that $Z\subset \cM_{g,n}(\HH_a,\beta)$ is an irreducible component dominating the target under the morphism
\begin{equation*}
    \tau:\cM_{g,n}(\HH_a,\beta)\to \cM_{g,n} \times \HH_a^n.
\end{equation*}
Suppose further that $n\ge 2$ (or that $g=0$ and $n\ge 3$). Then, the component $Z$ is generically smooth of the expected dimension.

In particular, if \eqref{dim_constraint} holds, then $\Tev^{\HH_a}_{g,n,\beta}$ is well-defined.
\end{lemma}

\begin{proof}
    Note that $\beta$ necessarily has strictly positive anti-canonical degree on $\HH_a$. By standard results on Severi varieties of surfaces (see e.g. \cite[Proposition 2.4]{cl_bn_surface}), a component $Z$ can only fail to be generically smooth of the expected dimension if every point of $Z$ parametrizes a map $f$ whose image has anti-canonical degree at most 3. The only possible images of such maps are the fibers of $\pi:\HH_a\to\P^1$ and the infinity-section. However, this is impossible if $f$ is required to interpolate through 2 or more general points of $\HH_a$.
\end{proof}

We prove Theorem \ref{tev_hirzebruch_zero} by analyzing the instability of tangent bundles of Hirzebruch surfaces. We may as well work in a slightly more general setting. Let 
\begin{equation*}
X=\P_{\P^{r-s}}(\cO\oplus\cO(-a_1)\oplus\cdots\oplus\cO(-a_{s}))    
\end{equation*}
be a $\P^{s}$-bundle over $\P^{r-s}$, where $0\le a_1\le\cdots\le a_{s}$ and $a_{s}>0$. Call $a=\sum_{i} a_i$. Let $\beta$ be an effective curve class on $X$. We write
\begin{equation*}
    \beta=\ell((c_1(\cO_{X/\P^{r-s}}(1)))^\vee+m((c_1(\cO_{\P^{r-s}}(1)))^\vee,
\end{equation*}
in the basis of $H_2(X)$ dual to $H^2(X)$. A map $f:C\to X$ in class $\beta$ given by:
\begin{itemize}
    \item a line bundle $\cM$ on $C$ of degree $m$, and a surjection $\cO^{r-s+1}\to \cM$, and
    \item a line bundle $\cL$ on $C$ of degree $\ell$, and a surjection $\cO \oplus \cM^{\otimes a_1} \oplus \cdots \oplus \cM^{\otimes a_s}\to\cL$.
\end{itemize}

\begin{lemma}\label{lemma: stability range}
In the notation above, the sub-bundle $T_{X/\P^{r-s}}\subset T_{X}$ is $\beta$-destabilizing if
\begin{equation*}
    \frac{(s+1)\ell-am}{s}>\frac{(s+1)\ell+(-a+r-s+1)m}{r}.
\end{equation*}

In particular, in all of these cases, if the geometric Tevelev degree $\Tev^{X}_{g,n,\beta}$ is well-defined, then it is zero.
\end{lemma}

\begin{proof}
 We have $\beta\cdot c_1(\cO_{\P^{r-s}}(1))=m$ and $\beta\cdot c_1(\cO_{X/\P^{r-s}}(1))=\ell$. In particular, we have $\beta\cdot T_{\P^{r-s}}=(r-s+1)m$. The relative Euler sequence
\begin{equation*}
    0\to \cO\to [\cO\oplus\cO(-a_1)\oplus\cdots\oplus\cO(-a_{s})]\otimes \cO_{X}(1) \to T_{X/\P^{r-s}} \to 0
\end{equation*}
shows that $\beta\cdot T_{X/\P^{r-s}}=(s+1)\ell-am$ and $\beta\cdot T_{X}=(s+1)\ell+(-a+r-s+1)m$. The conclusion follows.
\end{proof}

\begin{proof}[Proof of Theorem \ref{tev_hirzebruch_zero}]
Taking $s=1$ and $r=2$ in Lemma \ref{lemma: stability range} and $X=\HH_a$, we have that $T_{\HH_a/\P^1}$ is destabilizing whenever $\ell>(1+\frac{a}{2})m$. On the other hand, if $\beta$ is the class of an irreducible curve $f:C\to \HH_a$, then $\ell>am$, unless: 
    \begin{itemize}
        \item $\beta$ is equal to $m$ times the class of the zero-section $C_0\subset \HH_a$, in which case $\ell=am$, or
        \item $\beta$ is equal to $m$ times the class of the infinity-section $C_\infty\subset \HH_a$, in which case $\ell=0$.
    \end{itemize}
    Indeed, if $\ell\le ma$, then in order for $\cO\oplus\cM^{\otimes a}\to \cL$ to be surjective, either $\cM^{\otimes a}\to\cL$ is an isomorphism, or $\cO\to\cL$ is an isomorphism. 

    Therefore, when $a\ge3$ and $\beta$ is the class of an irreducible curve $f:C\to \HH_a$, then $T_{\HH_a}$ is $\beta$-unstable unless $\beta$ is a multiple of $[C_{\infty}]$, and if $a=2$, then $T_{\HH_a}$ is $\beta$-unstable unless $\beta$ is a multiple of $[C_{\infty}]$ or $[C_0]$. However, if $\beta$ is a multiple of $[C_{\infty}]$, then \eqref{dim_constraint} can never be satisfied. Theorem \ref{tev_hirzebruch_zero} now follows from Lemmas \ref{lem: tev well-defined} and \ref{lemma: stability}.
\end{proof}

As another example of Lemma \ref{lemma: stability range}, take $X=\Bl_q(\P^r)$, which is the case $s=1$ and $a=1$. Then, we have that $T_{X/\P^{r-1}}$ is destabilizing whenever 
    \begin{equation*}
     2\ell-m>\frac{2\ell+(r-1)m}{r},
    \end{equation*}
    or equivalently $(2r-1) m<(2r-2)\ell$. The vanishing of $\Tev^X_{g,n,\beta}$ in this range agrees with \cite[Theorem 13]{cl2}.


We proceed now to the proof of Theorem \ref{tev_h2}. Fix a general pointed curve $(C,p_1,\ldots,p_n)$ of genus $g$. We will use the following standard results.

\begin{lemma}\label{lemma: base-point free trick}
Let $\cM$ be a line bundle on $C$, and suppose that $V\subset H^0(C,\cM)$ is a base-point free subspace of dimension 2 (linear series of rank 1). Then, we have $H^1(C,\cM^{\otimes 2})=0$.
\end{lemma}

\begin{proof}
    This is the ``base-point-free pencil trick.'' The kernel of the multiplication map 
    \begin{equation*}
    V \otimes H^0(C, \omega_C \otimes \cM^{-1}) \to H^0(C,\omega_C)
    \end{equation*}
    is $H^0(C, \omega_C \otimes  \cM^{\otimes -2}) \cong H^1(C,\cM^{\otimes 2})^\vee$. By the Gieseker-Petri theorem, this map is injective, so we conclude that $H^1(C,\cM^{\otimes 2})=0$.
\end{proof}

\begin{lemma}\label{lem: h^0 twist is zero}
Let $W^1_m(C)\subset \Pic^m(C)$ be the locus of line bundles $\cM$ with $h^0(C,\cM)\ge 2$, and let $\cM \in W^1_m(C)$ be a general point. Then, we have
    \begin{equation*}
    H^0(C, \cM^{\otimes 2}(-p_1-\ldots -p_{2m-g+1}))=0.
    \end{equation*}
\end{lemma}

\begin{proof}
     The Brill-Noether theorem implies that the general rank 1 linear series on $C$ is base-point free. Thus, by Lemma \ref{lemma: base-point free trick} and Riemann-Roch, we have $H^0(C,\cM^{\otimes 2})=2m-g+1$ for the general $\cM \in W^1_m(C)$. Then, for general points $p_1,\ldots,p_{2m-g+1}$, twisting $\cM^{\otimes 2}$ down in succession decreases the dimension of the space of global sections. The Lemma follows.
\end{proof}

Fix general points $y_1,\ldots,y_n\in\P^1$. Recall that $\Tev^{\P^1}_{g,n,m}$ enumerates the number of degree $m$ covers $f:C\to \P^1$ sending the $p_i\in C$ to $y_i\in \P^1$. The assumption \eqref{dim_constraint} becomes $n=2m-g+1$; fix $g,n,m$ satisfying these conditions. If $n<3$, there are no degree $m$ maps $f:C\to \P^1$, by the Brill-Noether theorem, so assume that $n\ge3$. Then, the geometric Tevelev degree $\Tev^{\P^1}_{g,n,m}$ is well-defined, and computed in \cite{cps,fl,cl1}; we do not restate the formula here.

\begin{lemma}\label{lemma: H0(M) in Tev P1}
    Let $\cM$ be a line bundle on $C$ underlying a map $f:C\to\P^1$ enumerated by $\Tev^{\P^1}_{g,n,m}$. Then, we have
    \begin{equation*}
        H^0(C,\cM^{\otimes 2}(-p_1-\ldots - p_n))=0.
    \end{equation*}
\end{lemma}

\begin{proof} 
Let $\cM_{m}(C,\P^1)$ be the moduli space of degree $m$ maps $C\to\P^1$, and consider the evaluation map
\begin{equation*}
    \ev: \cM_m(C,\P^1) \times C^n \to (\P^1)^n \times C^n.
\end{equation*}
As $C$ is general, the space $\cM_{m}(C,\P^1)$ is pure of the expected dimension, and $\ev$ has relative dimension 0. Thus, the locus on $\cM_m(C,\P^1) \times C^n$ where the conclusion of Lemma \ref{lem: h^0 twist is zero} fails cannot dominate the target.
\end{proof}

\begin{proof}[Proof of Theorem \ref{tev_h2}]
A map $u: C \to \mathcal{H}_2$ in class $m [C_0]$ is given by:
\begin{itemize}
    \item a line bundle $\cM$ on $C$ of degree $m\ge0$, and a surjective map $\cO^{\oplus 2} 
    \to \cM$, and
    \item a line bundle $\cL$ on $C$ of degree $2m$, and a surjective map $\cM^{\otimes 2}\oplus\cO\to\cL$.
\end{itemize}
Write $[u_1:u_2]$ for the map $\cO^2\to \cM$ and $[u_0:u_3]$ for the map $\cM^{\otimes 2}\oplus\cO\to\cL$; we regard $u_0\in H^0(C, \cL\otimes\cM^{\otimes -2})$ and $u_3\in H^0(C,\cL)$. 

Two 4-tuples $(u_0,u_1,u_2,u_3)$ yield the same map if and only if they differ by the action of $(\C^*)^2$ given by  $(t_1,t_2) \cdot (u_0,u_1,u_2,u_3)=(t_2u_0,t_1 u_1, t_1u_2,t_1^2t_2u_3)$. For such a map to contribute to $\mathsf{Tev}^{\mathcal{H}_2}_{g,n,m}$ ($n>0$), it must be that $u_0 \neq 0$. Thus, we have $\cL\cong\cM^{\otimes 2}$, and that $u_0$ is a non-zero scalar. We may assume that $u_0=1$. 

Now, for general points $p_1,\ldots,p_n\in C$ and $y_1,\ldots,y_n\in \P^1$, the number of pairs of sections $u_1,u_{2}$ (up to scalar) for which $u'=[u_1:u_2]$ defines a map $g:C\to\P^1$ of degree $m$ with $g(p_i)=y_i$ is precisely $\Tev^{\P^1}_{g,n,m}$. Moreover, the maps $g$ have no infinitesimal deformations, and we have $H^0(C,\cL)=n$ and $H^0(\cL(-p_1-\ldots-p_n))=0$ by Lemmas \ref{lemma: base-point free trick} and \ref{lemma: H0(M) in Tev P1}, respectively.

We now fix the sections $u_1,u_2$ (up to scaling) defining the map $g:C\to\P^1$ as above. Let $\pi: \mathcal{H}_2= \P(\cO \oplus \cO(-2)) \to \P^1$ be the projection. We claim that, for general points $x_1,\ldots,x_n\in \mathcal{H}_2$ with the property that $\pi(x_i)=y_i$, there exists a unique map (with no infinitesimal deformations) $u:C\to \mathcal{H}_2$ in curve class $m [C_0]$ with the property that $u(p_i)=x_i$. To see this, write 
$$
\pi^{-1}(y_i)= \P( \cO|_{y_i} \oplus \cO(-2)|_{y_i}) \ni x_i=[t_i:v_i].
$$
where $t_i \in \C$ and $0 \neq  v_i \in \cO(-2)|_{y_i}$. Then, the condition $u(p_i)=x_i$ amounts to
\begin{equation}\label{eqn: u3 in H2}
    (u_3(p_i))(v_i)=t_i \in \C.
\end{equation}
Note that $u_3(p_i) \in \cL|_{p_i}=\cM^{\otimes 2}|_{p_i}=\cO(2)|_{y_i}$. This is a system of $n=\dim H^0(C,\cL)$ affine-linear equations on $H^0(C,\cL)$. When all of the $t_i$ are equal to zero, then $u_3=0$ is the unique solution to this system of equations, because $H^0(C,\cL(-p_1-\cdots -p_n))=0$. Therefore, the system of equations \eqref{eqn: u3 in H2} has a unique solution for any choice of $t_i$. The maps counted by $\Tev^{\HH_2}_{g,n,\beta}$ and $\Tev^{\P^1}_{g,n,m}$ are therefore in bijection.
\end{proof}

\bibliographystyle{plain} 
\bibliography{bib}

\end{document}